\documentclass[a4paper, reqno, 11pt]{amsart}
\usepackage[margin=1.4in]{geometry}
\usepackage{graphicx} 
\usepackage{amsthm,amsfonts,amssymb,amsmath,amsxtra,amsrefs}
\usepackage{mathrsfs}
\usepackage{graphicx}
\usepackage{tikz}
\usepackage{tikz-cd}
\usepackage{xr-hyper}
\usepackage{xr-hyper}
\usepackage[colorlinks=true, citecolor=blue]{hyperref}
\usepackage[noabbrev]{cleveref}
\usepackage{todonotes,cancel}
\usepackage[new]{old-arrows}

\newtheorem{thm}{Theorem}

\newtheorem{thmintro}{Theorem}

\newtheorem{prop}[thm]{Proposition}

\newtheorem{lemma}[thm]{Lemma}

\theoremstyle{definition}

\numberwithin{thm}{section}

\newcommand{\I}{\mathrm{I}}
\newcommand{\U}{\mathrm{U}}
\newcommand{\Qq}{\mathbb{C}(q^{1/2})}
\newcommand{\A}{\mathcal{A}}
\newcommand{\T}{\mathrm{T}}
\newcommand{\intU}{_\A\U}
\newcommand{\CU}{_{\mathbb{C}}\U}

\title{Quantum duality principle and quantum symmetric pairs}
\author[Jinfeng Song]{Jinfeng Song}
\address{Department of Mathematics, National University of Singapore, Singapore.}
\email{j\_song@u.nus.edu}

\begin{document}

\maketitle

\begin{abstract}

The \emph{quantum duality principle} (QDP) by Drinfeld predicts a connection between the \emph{quantized universal enveloping algebras} and the \emph{quantized coordinate algebras}, where the underlying classical objects are related by the duality in Poisson geometry. The current paper gives an explicit formulization of the QDP for quantum symmetric pairs.

Let $\mathfrak{g}$ be a complex semi-simple Lie algebra, equipped with the standard Lie bialgebra structure. Let $\theta$ be a Lie algebra involution on $\mathfrak{g}$ and denote by $\mathfrak{k}=\mathfrak{g}^\theta$ the fixed point subalgebra. The quantum symmetric pair $(\U,\U^\imath)$ is originally defined to be a quantization of the symmetric pair of the universal enveloping algebras $(U(\mathfrak{g}),U(\mathfrak{k}))$. In this paper, we show that an explicit specialisation of $(\U,\U^\imath)$ gives rise to the pair of the coordinate algebras $(\mathcal{O}(G^*),\mathcal{O}(K^\perp\backslash G^*))$, where $G^*$ is the dual Poisson-Lie group with the Lie algebra $\mathfrak{g}^*$, and $K^\perp\backslash G^*$ is a $G^*$-Poisson homogeneous space. Here $K^\perp$ is the closed subgroup of $G^*$associated to the complementary dual of $\mathfrak{k}$. Therefore $(\U,\U^\imath)$ can be viewed as a pair of quantized coordinate algebras. This generalises the well-known fact that the quantum group $\U$ provides a quantization of the coordinate algebra $\mathcal{O}(G^*)$.


\end{abstract}


\section{Introduction} 

\subsection{}
Let $G$ be a complex semi-simple group, equipped with the standard multiplicative Poisson structure. Let $\U=U_q\mathfrak{g}$ be the Drinfeld--Jimbo quantum group, which is a Hopf algebra over the base field $\Qq$. When viewed as the quantized universal enveloping algebra of $\mathfrak{g}$, the algebra $\U$ is a central object in representation theory. 

One can also view $\U$ as the quantized coordinate algebra on the dual Poisson group $G^*$ as follows. Let $\A=\mathbb{C}[q^{1/2},q^{-1/2}]$ be the subring of the base field $\Qq$. In comparison with Lusztig's integral form \cite{Lu93}*{3.1.13}, one can construct a non-standard $\A$-form $_\A\U$ of  $\U$ (see \S \ref{sec:rd}), where the base change $_\mathbb{C}\U=\mathbb{C}\otimes_\A{}_\A\U$ becomes a commutative algebra. Moreover $_\mathbb{C}\U$ carries a Poisson algebra structure where the Poisson brackets are induced by taking the commutators in $_\A\U$ (see \eqref{eq:poi}). 

The work of De Concini--Kac--Procesi\footnote{The work \cite{DCKP} deals with the quantum algebra at roots of unity. Similar proof can be applied to the current setting.} \cite{DCKP}*{Theorem 7.6} gives an explicit Poisson algebra isomorphism 
\begin{equation}\label{iso}
    \varphi:{}_\mathbb{C}\U\overset{\sim}{\longrightarrow}\mathcal{O}(G^*).
\end{equation}
We recall the map $\varphi$ in \S \ref{sec:sp}. Let us mention that although the QDP suggests that $\U$ shall provide a quantization of $\mathcal{O}(G^*)$, it takes much more efforts to explicitly establish the isomorphism \eqref{iso} (see \cite{DCKP}).


\subsection{}

Let $\theta$ be a complex group involution on $G$, and $K=G^\theta$ be the $\theta$-fixed point subgroup. Let $\mathfrak{g}$ and $\mathfrak{k}$ be the Lie algebras of $G$ and $K$, respectively. Associated with the symmetric pair $(\mathfrak{g},\mathfrak{k})$, the \emph{quantum symmetric pair} $(\U,\U^\imath)$ is the pair which consists of the quantum group $\U$, and a coideal subalgebra $\U^\imath$, called an \emph{$\imath$quantum group} \cites{Let02,BW18}. As a quantization of the universal enveloping algebra of $\mathfrak{k}$ inside $\U$, the algebra $\U^\imath$ has found important applications in representation theory and algebraic geometry \cites{BS22,BS24,BW}.

This paper study the $\imath$quantum group $\U^\imath$ from the another point of view of quantized coordinate algebras. Let $_\A\U^\imath=\U^\imath\cap{}_\A\U$ be the $\A$-form of $\U^\imath$ induced from the non-standard integral form on $\U$. This integral form is essentially different from the one studied by Bao and Wang \cite{BW18}. Let $_\mathbb{C}\U^\imath=\mathbb{C}\otimes_\A{}_\A\U^\imath$ be the commutative Poisson algebra obtained by the specialisation $q^{1/2}\mapsto 1$, and let $_\mathbb{C}\iota:{}_\mathbb{C}\U^\imath\rightarrow{}_\mathbb{C}\U$ be the Poisson algebra homomorphism obtained by the base change of the natural embedding $\iota:{}_\A\U^\imath\rightarrow{}_\A\U$. 

Let $\mathfrak{k}^\perp$ be the space of linear forms on $\mathfrak{g}$ which vanish on $\mathfrak{k}$. We shown (Proposition \ref{prop:coiso}) that $\mathfrak{k}^\perp$ is moreover a Lie subalgebra of $\mathfrak{g}^*$. Let $K^\perp$ be the connected closed subgroup of $G^*$ associated with $\mathfrak{k}^\perp$. The affine quotient $K^\perp\backslash G^*$ is automatically a $G^*$-Poisson homogeneous space, whose coordinate algebra is isomorphic to the Poisson subalgebra of $K^\perp$-invariant functions on $G^*$.


The main theorem of the paper is the following. 

\begin{thmintro}\label{main}
There is a unique Poisson algebra isomorphism
\begin{equation}
    \varphi^\imath:{}_\mathbb{C}\U^\imath\overset{\sim}{\longrightarrow}{}\mathcal{O}(K^\perp\backslash G^*),
\end{equation} 
such that the diagram
    \begin{equation}
        \begin{tikzcd}
            & \mathcal{O}(K^\perp\backslash G^*) \arrow[r,hook] & \mathcal{O}(G^*) \\
            & _\mathbb{C}\U^\imath \arrow[u,"\varphi^\imath"] \arrow["_\mathbb{C}\iota"',r] & {}_\mathbb{C}\U \arrow[u,"\varphi"']
        \end{tikzcd}
    \end{equation}
    commutes.     
\end{thmintro}

Our construction of $\varphi^\imath$ is explicit, and hence it is suitable for applications.
 

\subsection{}

The Poisson structures $K^\perp\backslash G^*$ occur in many fields of mathematics and physics (\cites{Bo,Bon,BGM}). Therefore our result provides a bridge to utilize the theory of quantum symmetric pairs to study such Poisson structures.  As an example, the Dubrovin--Ugaglia Poisson structures on the space of $(n\times n)$-upper triangular matrices with 1 on the diagonal appear in the study of Frobenius manifolds \cites{Du,Ug}, and are related to the connections on the Riemann surfaces as well as Poisson--Lie groups by the work of Boalch \cites{Bo,Bo2}. In \S \ref{ap} we see that the $\imath$quantum group associated to the symmetric pair $(\mathfrak{sl}_n,\mathfrak{so}_n)$ is a quantization of the Dubrovin--Ugaglia Poisson structure, and the well-known braid group symmetries on such Poisson structure can be obtained from the braid group symmetries on the corresponding $\imath$quantum group via specialisation.  



Viewing $\U$ as the quantized coordinate algebra is the fundamental philosophy to relate cluster algebras with quantum groups \cites{SS,Sh22}. Based on the same point of view, in the work \cite{So} the author gave a cluster algebra realisation of the $\imath$quantum group associated to $(\mathfrak{sl}_n,\mathfrak{so}_n)$. It is desirable to generalise the construction to other types. Our result provides the geometric foundation in this direction.

\subsection{}
In \cites{CG06}, Ciccoli and Gavarini generalised Drinfeld's QDP to the setting of subgroups and homogeneous spaces. However, like the work of Drinfeld \cite{Dr86}, the construction in \cites{CG06} is based on the formal power series setting and hence cannot be applied to our situation. In their subsequent paper \cite{CG14}, by imposing various technical assumptions, the authors deal with the quantum algebras over the ordinary fractional field, which they call the global version of the QDP. Their results are general but abstract, and it is difficult to check those assumptions for explicit examples. Our result fits into the general philosophy of their QDP, but our proof is independent of theirs.

\vspace{.2cm}\noindent {\bf Acknowledgement: }The author is supported by Huanchen Bao’s MOE grant A-0004586-00-00 and A-0004586-01-00.

\section{Preliminaries}\label{sec:pre}

In this section, we recall basic construction on quantum groups, quantum symmetric pairs, (dual) Poisson groups and the specialisations of quantum algebras.

\subsection{Quantum groups}\label{sec:rd}

Fix a finite index set $\I$, and a Cartan matrix $A=(a_{ij})_{i,j\in\I}$. Let $D=diag(\epsilon_i)_{i\in\I}$ be the diagonal matrix such that $DA$ is symmetric positive definite, with $\epsilon_i\in\mathbb{Z}_{>0}$ and $\{\epsilon_i\mid i\in\I\}$ relatively prime. 

Let $\mathfrak{g}$ be the complex semi-simple Lie algebra associated with the Cartan matrix $A$, and $\mathfrak{h}$ be a Cartan subalgebra of $\mathfrak{g}$. Take Borel subalgebras $\mathfrak{b}^+$ and $\mathfrak{b}^-$ of $\mathfrak{g}$, such that $\mathfrak{b}^+\cap\mathfrak{b}^-=\mathfrak{h}$. Denote by $\Phi\subset \mathfrak{h}^*$ the root system of $\mathfrak{g}$ relative to $\mathfrak{h}$. Let $\mathcal{R}^+\subset\Phi$ be the set of positive roots associated with $\mathfrak{b}^+$, and $\Delta=\{\alpha_i\mid i\in\I\}$ be the subset of simple roots. Set $\mathcal{R}^-=-\mathcal{R}^+$. Let $$\mathfrak{g}=\mathfrak{h}\oplus\bigoplus_{\alpha\in\Phi}\mathfrak{g}_\alpha$$ be the root space decomposition. Set $\mathfrak{u}^+=\oplus_{\alpha\in\mathcal{R}^+}\mathfrak{g}_\alpha$ and $\mathfrak{u}^-=\oplus_{\alpha\in\mathcal{R}^-}\mathfrak{g}_{\alpha}$. Let $Q$ be the root lattice, and $W$ be the Weyl group of $\mathfrak{g}$ with generators $\{s_i\mid i\in\I\}$. The group $W$ acts on $Q$ in a natural way. For an element $w$ in $W$, let $|w|$ be the length of $w$. Let $w_0$ be the longest element in $W$, and write $n=|w_0|$.






Let $q$ be an indeterminate, and $q^{1/2}\in\overline{\mathbb{Q}(q)}$ be a fixed square root of $q$. Let $q_i=q^{\epsilon_i}$, for $i\in\I$. Let $\U$ be the Drinfeld--Jimbo quantum group associated with $\mathfrak{g}$ over the base field $\Qq$ (cf. \cite{Jan96}*{4.3}). Let $\varepsilon:\U\rightarrow\Qq$ be the counit. For $i\in\I$, let $\mathbf{E}_i$, $\mathbf{F}_i$ and $\mathbf{K}_i$ be the generators as in \emph{loc. cit.} For our purpose, we consider the rescaled generators 
$$E_i= q_i^{1/2}(q_i^{-1}-q_i)\mathbf{E}_i,\quad F_i= q_i^{-1/2}(q_i-q_i^{-1})\mathbf{F_i},\quad K_i= \mathbf{K}_i,\quad \text{for }i\in\I.$$

	


Let $\U^-$ (resp., $\U^+$) be the unital $\Qq$-subalgebra of $\U$ generated by $F_i$ (resp., $E_i$), for various $i\in \I$. Let $\U^0$ be the unital $\Qq$-subalgebra of $\U$ generated by $K_i^{\pm1}$, for $i\in \I$. For any $\mu=\sum_{i\in\I}a_i\alpha_i\in Q$, we write $K_\mu=\prod_{i\in\I}K_i^{a_i}$. 

For $i\in\I$, let $\T_i$ be the $\Qq$-algebra automorphism of $\U$, denoted by $T_{i,-1}''$ in \cite{Lu93}*{37.1.3}. It is known that automorphisms $\{\T_i\mid i\in\I\}$ satisfy the braid group relations. For $w\in W$, we write $\T_w=\T_{i_1}\T_{i_2}\cdots\T_{i_r}$, where $w=s_{i_1}s_{i_2}\cdots s_{i_r}$ and $r=|w|$.

Let $\mathbf{i}=(i_1,i_2,\cdots,i_n)\in \I^n$ be a reduced expression of $w_0$. For $\mathbf{a}=(a_1,\cdots,a_n)\in\mathbb{N}^n$, set 
\begin{equation*}
\begin{split}
    &F_{\mathbf{i}}(\mathbf{a})=E_{i_1}^{a_1}\T_{i_1}(E_{i_2}^{a_2})\cdots \T_{i_1}\cdots\T_{i_{n-1}}(E_{i_n}^{a_n}),\\
    &E_{\mathbf{i}}(\mathbf{a})=F_{i_1}^{a_1}\T_{i_1}(F_{i_2}^{a_2})\cdots \T_{i_1}\cdots\T_{i_{n-1}}(F_{i_n}^{a_n}).
\end{split}
\end{equation*}
For $1\leq k\leq n$, we write $F_{\mathbf{i},k}=F_{\mathbf{i}}(\mathbf{e}_k)$ and $E_{\mathbf{i},k}=E_{\mathbf{i}}(\mathbf{e}_k)$, where $\mathbf{e}_k=(\delta_{kj})_{j=1}^n\in\mathbb{N}^n$.

Set $\A=\mathbb{C}[q^{1/2},q^{-1/2}]$ to be the subring of $\Qq$. Let $\intU^-$ (resp., $\intU^+$) be the $\A$-submodule of $\U^-$ (resp., $\U^+$), spanned by $F_{\mathbf{i}}(\mathbf{a})$ (resp., $E_{\mathbf{i}}(\mathbf{a})$), for various $\mathbf{a}\in\mathbb{N}^n$. It is known that $\intU^-$ and $\intU^+$ are independent of the choice of $\mathbf{i}$, and are closed under multiplications. Let $\intU^0$ be the unital $\A$-subalgebra of $\U^0$, generated by $K_i^{\pm1}$, for $i\in\I$. Let $\intU$ be the $\A$-subalgebra of $\U$, generated by $\intU^-$, $\intU^0$ and $\intU^+$. It is different from the integral form defined by Lusztig \cite{Lu93}*{3.1.13}, and is the same as the one studied by Berenstein--Greenstein \cite{BG17}. By \cite{BG17}*{Corollary 3.13 \& Theorem 3.11}, the multiplication gives an isomorphism as $\A$-modules
\begin{equation}\label{eq:tt}
\intU\cong{} \intU^+\otimes{}{\intU^0}\otimes{} {\intU^-}.
\end{equation}
Here the tensor products are over $\A$. It is easy to see that the $\A$-subalgebra $_\A\U$ is invariant under the braid group action $\T_i$, for $i\in\I$.

\subsection{Quantum symmetric pairs}\label{sec:SD}

Let $(\I=\I_{\circ}\sqcup\I_{\bullet},\tau)$ be a Satake diagram (cf. \cite{BW18}). Recall that $\tau$ is a graph involution on $\I$, which leaves $\I_\circ$ and $\I_\bullet$ invariant. We fix once for all a subset $\I'_\circ\subset \I_\circ$, consisting of exactly one element in each $\tau$-orbits of $\I_\circ$. Let $\Phi_\bullet\subset\Phi$ be the sub-root system with simple roots $\{\alpha_i\mid i\in\I_\bullet\}$, and $W_\bullet=\langle s_i\mid i\in\I_\bullet\rangle$ be the parabolic subgroup. Let $w_\bullet$ be the longest element in $W_\bullet$. Let $\mathcal{R}_\bullet^+=\mathcal{R}^+\cap\Phi_\bullet$ and $\mathcal{R}_\bullet^-=\mathcal{R}^-\cap\Phi_\bullet$. Let $\theta=-w_\bullet\tau:Q\rightarrow Q$ be the group involution on $Q$, and $Q^\theta=\{\mu\in Q\mid \theta(\mu)=\mu\}$. 


We fix parameters $\varsigma_i\in \pm q^{\mathbb{Z}}$, for $i\in\I_\circ$, which satisfy the conditions in \cite{BW18}*{Definition 3.5}. The associated \emph{$\imath$quantum group} $\U^\imath$ is the unital $\Qq$-subalgebra of $\U$, generated by the following elements,
\begin{align*}
&B_i=F_i-\varsigma_i\T_{w_\bullet}(E_{\tau i})K_i^{-1},\; k_i=K_iK_{\tau i}^{-1},\qquad\text{for } i\in \I_\circ;\\&  E_i,\;F_i,\;K_i^{\pm1}\qquad\text{for } i\in \I_\bullet.
\end{align*}
We write $B_i=F_i$, for $i\in\I_\bullet$. The pair $(\U,\U^\imath)$ is called a \emph{quantum symmetric pair}.

\subsection{Poisson structures}\label{eq:ngp}

For any complex affine algebraic variety $V$, we write $\mathcal{O}(V)$ to denote the $\mathbb{C}$-algebra of regular functions on $V$. 

Let $G$ be the complex semi-simple adjoint group with the Lie algebra $\mathfrak{g}$. Let $B^+$, $B^-$ and $H$ be the connected closed subgroups of $G$, whose Lie algebras are $\mathfrak{b}^+$, $\mathfrak{b}^-$ and $\mathfrak{h}$, respectively. Let $U^+$ (resp., $U^-$) be the unipotent radical of $B^+$ (resp., $B^-$). For $\alpha\in\Phi$, let $U_\alpha$ be the root subgroup of $G$ associated with $\alpha$.


Let us briefly recall the standard Poisson structure associated with $G$, and the dual Poisson group $G^*$. We refer to \cite{CP}*{\S 1} for a detailed exposition. Let $\langle\,,\,\rangle$ be the Killing form of $\mathfrak{g}$, and $\mathfrak{g}^*$ be the space of $\mathbb{C}$-linear forms on $\mathfrak{g}$. The \emph{Drinfeld double} $\mathfrak{d}=\mathfrak{g}\oplus\mathfrak{g}$ admits a non-degenerate invariant bilinear form $\langle\langle\,,\,\rangle\rangle$, given by 
\begin{equation*}
    \langle\langle(X_1,Y_1),(X_2,Y_2)\rangle\rangle=\langle X_1,X_2\rangle-\langle Y_1,Y_2\rangle,\quad \forall\, X_1,X_2,Y_1,Y_2\in\mathfrak{g}.
\end{equation*}

Let us identify $\mathfrak{g}$ with the diagonal Lie subalgebra of $\mathfrak{d}$, and identify $\mathfrak{g}^*$ with another Lie subalgebra,
\begin{equation}\label{eq:dg}
    \mathfrak{g}^*\cong\{(X,Y)\in\mathfrak{b}^+\times\mathfrak{b}^-\mid \text{the $\mathfrak{h}$-component of $X+Y$ is 0}\},
\end{equation}
via the bilinear form $\langle\langle\,,\,\rangle\rangle$. Under these identifications the triple $(\mathfrak{d},\mathfrak{g},\mathfrak{g}^*)$ forms a \emph{Manin triple}, which determines a \emph{Lie bialgebra} structure on $\mathfrak{g}$, and determines the \emph{standard Poisson structure} on $G$. 

By a dual construction, one gets a Lie bialgebra structure on $\mathfrak{g}^*$, which determines a Poisson structure on the \emph{dual Poisson group} $G^*$, where $$G^*=\{(u_+t,t^{-1}u_-)\mid u_\pm\in U^\pm,t\in H\}$$ is the connected closed subgroup of $B^+\times B^-$ with the Lie algebra $\mathfrak{g}^*$.

We have an isomorphism as varieties
\begin{equation}\label{igd}
    U^+\times H\times U^-\overset{\sim}{\longrightarrow} G^*,\qquad (u_+,t,u_-)\mapsto (u_+t,t^{-1}u_-),
\end{equation}
which determines an isomorphism as $\mathbb{C}$-algebras
\begin{equation}\label{eq:og}
    \mathcal{O}(G^*)\cong\mathcal{O}(U^+)\otimes\mathcal{O}(H)\otimes\mathcal{O}(U^-).
\end{equation}

\subsection{Specialisation of quantum algebras}\label{eq:sqa}

Let $R_q$ be a non-commutative unital $\A$-algebra, which satisfies the condition
\begin{equation}\label{eq:con}
    [R_q,R_q]\subseteq (q^{1/2}-1)R_q.
\end{equation}
Let $R=\mathbb{C}\otimes_{\A}R_q$ where the tensor product is given by $q^{1/2}\mapsto 1$. The condition \eqref{eq:con} is equivalent to saying that the $\mathbb{C}$-algebra $R$ is commutative. Denote by $f\mapsto \overline{f}$ the canonical map $R_q\rightarrow R$. Then $R$ carries a Poisson bracket $\{\,,\,\}:R\times R\rightarrow R$, defined by 
\begin{equation}\label{eq:poi}
    \{\overline{f},\overline{g}\}=\overline{\frac{1}{2(q^{1/2}-1)}(fg-gf)},\quad \forall f,g\in R_q.
\end{equation}

\subsection{Specialisation of $_\A\U$}\label{sec:sp}
Let us fix a pinning $\{x_i,y_i\mid i\in \I\}$ of the group $G$, where $x_i:\mathbb{C}\rightarrow U_{\alpha_i}$ and $y_i:\mathbb{C}\rightarrow U_{-\alpha_i}$ are one-parameter subgroups.  
For $i\in\I$, set 
\begin{equation*}
    \dot{s_i}=x_i(1)y_i(-1)x_i(1)\in G.
\end{equation*}
It is well-known that elements $\{\dot{s}_i\mid i\in\I\}$ satisfy the braid group relations. For $w\in W$, denote by $\dot{w}=\dot{s}_{i_1}\cdot\dot{s}_{i_2}\cdots\dot{s}_{i_k}$, for one (and hence for all) reduced expression $w=s_{i_1}s_{i_2}\cdots s_{i_k}$.

Take a reduced expression $\mathbf{i}=(i_1,\cdots, i_n)$ of the longest $w_0$ of $W$. For $1\leq k\leq n$, set $\beta_{\mathbf{i},k}=s_{i_1}\cdots s_{i_{k-1}}(\alpha_{i_k})\in \mathcal{R}^+$. Set $x_{\mathbf{i},k}=\text{Ad}_{\dot{s}_{i_1}\cdots\dot{s}_{i_{k-1}}}\circ x_{i_k}:\mathbb{C}\rightarrow U_{\beta_{\mathbf{i},k}}$, and $y_{\mathbf{i},k}=\text{Ad}_{\dot{s}_{i_1}\cdots\dot{s}_{i_{k-1}}}\circ y_{i_k}:\mathbb{C}\rightarrow U_{-\beta_{\mathbf{i},k}}$. We have $x_{\mathbf{i},k}=x_j$ and $y_{\mathbf{i},k}=y_j$, if $\beta_{\mathbf{i},k}=\alpha_j$, for some $j\in\I$.

By \cite{Spr09}*{Lemma 8.3.5}, elements $a\in U^+$, $b\in U^-$ can be factored as
\begin{equation}\label{eq:fac}
    a=x_{\mathbf{i},n}(a_n)x_{\mathbf{i},n-1}(a_{n-1})\cdots x_{\mathbf{i},1}(a_1)\quad \text{and}\quad b=y_{\mathbf{i},1}(b_1)y_{\mathbf{i},{2}}(b_{2})\cdots y_{\mathbf{i},n}(b_n),
\end{equation}
for unique tuples $(a_1,\cdots,a_n)\in\mathbb{C}^n$ and $(b_1,\cdots,b_n)\in\mathbb{C}^n$. We define $\chi_{\mathbf{i},k}^+\in\mathcal{O}(U^+)$ by $\chi_{\mathbf{i},k}^+(a)=a_k$, and define $\chi_{\mathbf{i},k}^-\in\mathcal{O}(U^-)$ by $\chi_{\mathbf{i},k}(b)=b_k$. If $\beta_{\mathbf{i},k}=\alpha_j$ is a simple root, elements $\chi^+_{\mathrm{i},k}$ and $\chi_{\mathrm{i},k}^-$ are independent of the reduced expression, in which case we denote by $\chi^+_j=\chi^+_{\mathrm{i},k}$ and $\chi_j^-=\chi_{\mathrm{i},k}^-$.

Recall the $\A$-algebras $_\A\U$, $_\A\U^+$, $_\A\U^-$ and $_\A\U^0$ in \S \ref{sec:rd}. Let $_\mathbb{C}\U$, $_\mathbb{C}\U^+$, $_\mathbb{C}\U^-$ and $_\mathbb{C}\U^0$ be the $\mathbb{C}$-algebras obtained by the base change $q^{1/2}\mapsto 1$ of the corresponding $\A$-algebras, which are known to be commutative \cite{BG17}*{Corollary 3.15}. Moreover the isomorphism \eqref{eq:tt} induces the isomorphism $_\mathbb{C}\U\cong{}_\mathbb{C}\U^+\otimes{}_\mathbb{C}\U^0\otimes{}_\mathbb{C}\U^-$ as $\mathbb{C}$-algebras. Here tensor products are over $\mathbb{C}$.

Take a reduced expression $\mathbf{i}$ of $w_0$. One has $\mathbb{C}$-algebra isomorphisms
\begin{equation}\label{eq:sp}
    \varphi^+:{}_\mathbb{C}\U^+\overset{\sim}{\longrightarrow}\mathcal{O}(U^+)\qquad\text{and}\qquad\varphi^-:{}_\mathbb{C}\U^-\overset{\sim}{\longrightarrow}\mathcal{O}(U^-),
\end{equation}
given by $\varphi^+(\overline{E_{\mathbf{i},k}})=\chi_{\mathbf{i},k}^+$ and $\varphi^-(\overline{F_{\mathbf{i},k}})=\chi_{\mathbf{i},k}^-$, for $1\leq k\leq n$. In particular $\varphi^+(\overline{E_i})=\chi_i^+$ and $\varphi^-(\overline{F_i})=\chi_i^-$, for $i\in\I$. The maps $\varphi^\pm$ are independent of the choice of $\mathbf{i}$.

Since $G$ has trivial center, one has the canonical $\mathbb{C}$-algebra isomorphism $\mathcal{O}(H)\cong \mathbb{C}[Q]$, which induces the $\mathbb{C}$-algebra isomorphism $\varphi^0:{}_\mathbb{C}\U^0\overset{\sim}{\longrightarrow}\mathcal{O}(H)$, given by $\varphi^0(\overline{K_\mu})=\mu$, for $\mu\in Q$.




The following theorem asserts that $_\A\U$ specializes to the coordinate algebra $\mathcal{O}(G^*)$. The proof essentially follows from the work of De Concini--Kac--Procesi.

\begin{thm}[\cite{DCKP}*{Theorem 7.6}]\label{thm:phi}
The map 
\begin{equation}\label{eq:phi}
    \varphi=\varphi^+\otimes\varphi^0\otimes\varphi^-:{\CU}\overset{\sim}{\longrightarrow}\mathcal{O}(G^*)
\end{equation}
is an isomorphism as Poisson algebras.
    
\end{thm}

Since $\intU$ is invariant under the action $\T_i$, for $i\in\I$, each $\T_i$ descends to a Poisson algebra automorphism on $_\mathbb{C}\U\cong\mathcal{O}(G^*)$. We use the same notation $\T_i$ to denote the induced action.

\section{Proof of the main theorem}\label{proof}

We give the proof of Theorem \ref{main} in this section. We need some preparation before final proof. Retain the notations in the previous section.

\subsection{Geometric preparations}\label{gp}
For $i\in \I_\circ$, recall the parameter $\varsigma_i$ in \S \ref{sec:SD}. Set $c_i=\overline{\varsigma_i}\in\{\pm1\}$. By \cite{Spr}, there is a unique complex group involution $\theta:G\rightarrow G$, such that,
\begin{align*}
    \theta(x_i(\xi))=x_i(\xi),\quad \theta(y_i(\xi))=y_i(\xi),\qquad\text{for $i\in\I_\bullet$, and $\xi\in\mathbb{C}$;}\\
    \theta(x_i(\xi))=\text{Ad}_{\dot{w}_\bullet}(y_{\tau i}(c_i\xi)),\quad \theta(y_i(\xi))=\text{Ad}_{\dot{w}_\bullet}(x_{\tau i}(c_i\xi)),\qquad \text{for $i\in\I_\circ$, and $\xi\in\mathbb{C}$.}
\end{align*}

Let $K=G^\theta$ be the fixed-point subgroup, and $\mathfrak{k}$ be the Lie algebra of $K$. We use the same notation $\theta$ to denote the induced Lie algebra involution on $\mathfrak{g}$. 

\begin{prop}\label{prop:coiso}
    The closed subgroup $K$ is coisotropic, that is, the complementary dual 
    \[
    \mathfrak{k}^\perp=\{f\in\mathfrak{g}^*\mid f(X)=0,\;\forall X\in\mathfrak{k}\}
    \]
    is a Lie subalgebra of $\mathfrak{g}^*$.
\end{prop}

\begin{proof}
We identify $\mathfrak{g}^*$ with the subalgebra of the double $\mathfrak{d}$ as in \eqref{eq:dg}. We \emph{claim} that under the identification one has
\begin{equation}
    \mathfrak{k}^\perp=\{(X_1,X_2)\in\mathfrak{g}^*\mid X_1=\theta(X_2)\}.
\end{equation}

Let $\mathfrak{p}=\{X\in\mathfrak{g}\mid \theta(X)=-X\}$. Then $\mathfrak{g}=\mathfrak{k}\oplus\mathfrak{p}$ as vector spaces. Since the involution $\theta:\mathfrak{g}\rightarrow\mathfrak{g}$ preserves the Killing form, we have $\mathfrak{p}=\{X\in\mathfrak{g}\mid \langle X,Y\rangle=0,\;\forall \;Y\in\mathfrak{k}\}$.

Take any $x=(X_1,X_2)\in\mathfrak{g}^*$. We have
\begin{align}
    x\in\mathfrak{k} &\iff 0=\langle\langle x,y\rangle\rangle=\langle X_1-X_2,Y\rangle,\quad \text{for any $y=(Y,Y)$ with $Y\in\mathfrak{k}$},\notag\\
    & \iff X_1-X_2\in\mathfrak{g}_1,\notag\\
    & \iff \theta(X_1)-X_2=\theta(X_2)-X_1.\label{eq:thx}
\end{align}
Let us write $X_1=X_1'+h$ and $X_2=X_2'-h$, where $X_1'\in\mathfrak{u}^+$, $X_2'\in\mathfrak{u}^-$ and $h\in\mathfrak{h}$. Then the condition \eqref{eq:thx} holds if and only if $\theta(X_1')-X_2'=\theta(X_2')-X_1'$ and $\theta(h)=-h$. Therefore $\theta(X_1')-X_2'\in\mathfrak{p}$. On the other hand, by the construction of $\theta$, we have $$\theta(X_1')-X_2'\in\bigoplus_{\alpha\in\mathcal{R}^--\Phi_\bullet}\mathfrak{g}_\alpha.$$
Note that $\mathfrak{g}_\alpha\subset\mathfrak{k}$, for $\alpha\in\Phi_\bullet$, and $\theta(\mathcal{R}^--\Phi_\bullet)\cap\mathcal{R}^-=\emptyset$. We conclude that $\theta(X_1')-X_2'=0$. Therefore we have $\theta(X_1)=X_2$. This proves the claim.

The proposition follows immediately.
\end{proof}

Let $K^\perp$ be the connected closed subgroup of $G^*$ with the Lie algebra $\mathfrak{k}^\perp$. It is also a coisotropic subgroup. Hence the quotient $K^\perp\backslash G^*$ is automatically a Poisson homogeneous space of $G^*$ (\cite{Dr93}). 

In order to study the geometry of $K^\perp\backslash G^*$, let us introduce another closed subgroup  
\[
P=\{(u_+t,t^{-1}u_-)\mid u_+\in U^+_{w_\bullet},u_-\in U^-,t\in H'\}\subset G^*,
\]
where $U_{w_\bullet}^+$ is the closed subgroup of $U^+$ with the Lie algebra $\oplus_{\alpha\in\mathcal{R}_\bullet^+}\mathfrak{g}_\alpha$, and $$H'=\{t\in H\mid \alpha_i(t)=1,\,\forall i\in \I_\circ'\}.$$ 

It follows immediately from the definition that the map
\begin{equation}\label{eq:P}
U_{w_\bullet}^+\times H'\times U^-\overset{\sim}{\longrightarrow} P,\qquad (u_+,t,u_-)\mapsto (u_+t,t^{-1}u_-),
\end{equation}
defines an isomorphism as varieties.


\begin{prop}\label{prop:dec}
The map 
\begin{equation}\label{eq:dec}
K^\perp\times P\overset{\sim}{\longrightarrow} G^*,\qquad (g_1,g_2)\mapsto g_1g_2,
\end{equation}
is an isomorphism as varieties. Therefore we have $K^\perp\backslash G^*\cong P$ as varieties. 
\end{prop}

\begin{proof}
    Let  $U_{w_0w_\bullet}^+$ be the closed subgroups of $U^+$, with the Lie algebra $\oplus_{\alpha\in\mathcal{R}^+-\mathcal{R}_\bullet^+}\mathfrak{g}_\alpha$. Let $H_1$ be the closed connected subgroup of $H$, with the Lie algebra $\{X\in\mathfrak{h}\mid\theta(X)=-X\}$. By the proof of Proposition \ref{prop:coiso}, the subgroup $K^\perp$ is the identity component of the subgroup 
$$\widetilde{K}=\{(b_+,b_-)\in G^*\mid b_+=\theta(b_-)\}.$$ Also note that  
    $U^+_{w_\circ w_\bullet}=U^+\cap \theta(U^-).$
Therefore the map 
\begin{equation}\label{eq:deo}
U^+_{w_0w_\bullet}\times H_1 \overset{\sim}{\longrightarrow} K^\perp,\qquad (u_+,t)\mapsto (u_+t,\theta(u_+t)),
\end{equation}
defines an isomorphism as varieties.

By \cite{Spr09}*{Lemma 8.3.5} the group multiplication $U^+_{w_0w_\bullet}\times U^+_{w_\bullet}\rightarrow U^+$ is an isomorphism as varieties. Since $G$ is of adjoint type, the multiplication $H_{1}\times H'\rightarrow H$ also defines an isomorphism as varieties. 

Since the adjoint action of $H$ preserves the subgroups $U^+_{w_0w_\bullet}$ and $U^+_{w_\bullet}$, the proposition then follows from the isomorphism \eqref{igd}.
\end{proof}

By Proposition \ref{prop:dec}, the quotient variety $K^\perp\backslash G^*$ is actually affine. Therefore by \cite{Spr09}*{Exercise 5.5.9 (8)}, the coordinate ring $\mathcal{O}(K^\perp\backslash G^*)$ is isomorphic to the subalgebra 
\[\{f\in\mathcal{O}(G^*)\mid f(kg)=f(g),\;\forall\,k\in K^\perp,g\in G^*\}\]
of $K^\perp$-invariant functions in $\mathcal{O}(G^*)$. We shall make this identification throughout the paper.

Let $\iota^*:\mathcal{O}(G^*)\rightarrow\mathcal{O}(P)$ be the comorphism of the embedding $\iota:P\hookrightarrow G^*$. Thanks to Proposition \ref{prop:dec}, we conclude that the restriction map
\begin{equation}\label{eq:pro}
    \iota^*\mid_{\mathcal{O}(K^\perp\backslash G^*)}:\mathcal{O}(K^\perp\backslash G^*)\overset{\sim}{\longrightarrow}\mathcal{O}(P)
\end{equation}
defines an isomorphism as $\mathbb{C}$-algebras.

In light of the isomorphism \eqref{eq:pro}, the map $\iota^*$ provides a tool to control the size of the algebra $\mathcal{O}(K^\perp\backslash G^*)$. We also need a similar map on the quantum level.

\subsection{Quantization of the group $P$}

We firstly describe a quantization of the algebra $\mathcal{O}(P)$. Following \cite{KY}*{2.3}, define the \emph{partial parabolic subalgebra} $\U_P$ to be the unital $\Qq$-subalgebra of $\U$, generated by $F_i$ ($i\in\I$), $E_i$ ($i\in \I_\bullet$) and $K_\mu$ ($\mu\in Q^\theta$). Let $\U^+_{w_\bullet}$ be the unital subalgebra of $\U^+$ generated by $E_i$, for various $i\in\I_\bullet$. Let $\U^{0\theta}$ be the unital subalgebra of $\U^0$ generated by $K_\mu$, for $\mu\in Q^\theta$. Then the multiplication gives an isomorphism as $\Qq$-vector spaces
\begin{equation}\label{eq:upd}
\U_P\cong \U_{w_\bullet}^+\otimes \U^{0\theta}\otimes \U^-.
\end{equation}

Recall the $\A$-subalgebra $\intU$ in \S \ref{sec:rd}. Set $_\A\U_P=\U_P\cap{}_\A\U$. Similarly define $_\A\U^{0\theta}$ and $_\A\U_{w_\bullet}^+$. By the PBW-bases, the multiplication gives an isomorphism as $\A$-modules
\begin{equation}\label{eq:dea}
    _\A\U_P\cong {}_\A\U_{w_\bullet}^+\otimes {}_\A\U^{0\theta}\otimes{}_\A\U^-,
\end{equation}
where the tensor products are over $\A$.

Let $_\mathbb{C}\U_P=\mathbb{C}\otimes {}_\A\U_P$ be the commutative $\mathbb{C}$-algebra obtained by the base change $q^{1/2}\mapsto1$. Similarly define the $\mathbb{C}$-algebras $_\mathbb{C}\U^+_{w_\bullet}$, $_\mathbb{C}\U^{0\theta}$ and $_\mathbb{C}\U^-$. Thanks to \eqref{eq:dea}, we have the isomorphism as $\mathbb{C}$-algebras
\begin{equation}\label{eq:is}
    _\mathbb{C}\U_P\cong {}_\mathbb{C}{\U}^+_{w_\bullet}\otimes {}_\mathbb{C}\U^{0\theta}\otimes {}_\mathbb{C}\U^-.
\end{equation}

Recall the $\mathbb{C}$-algebra isomorphisms $\varphi^\pm:{}_\mathbb{C}\U^\pm\overset{\sim}{\rightarrow}\mathcal{O}(U^\pm)$ in \eqref{eq:sp}. The map $\varphi^+$ induces an isomorphism 
$
\varphi^+_{w_\bullet}:{}_\mathbb{C}\U_{w_\bullet}^+\overset{\sim}{\rightarrow}\mathcal{O}(U_{w_\bullet}^+).
$ Let
$
\varphi^0: {}_\mathbb{C}\U^{0\theta}\overset{\sim}{\rightarrow}\mathcal{O}(H')
$
be the $\mathbb{C}$-algebra isomorphism given by $\overline{k_i}\mapsto \alpha_i\mid_{H'}$, for $i\in \I_\circ-\I_\circ'$, and $\overline{K_i}\mapsto \alpha_i\mid_{H'}$, for $i\in \I_\bullet$.

By \eqref{eq:P}, we have the natural isomorphism $\mathcal{O}(P)\cong\mathcal{O}(U^+_{w_\bullet})\otimes \mathcal{O}(H')\otimes \mathcal{O}(U^-)$ as $\mathbb{C}$-algebras. Under this isomorphism, one gets the $\mathbb{C}$-algebra isomorphism
\[
\varphi_P=\varphi_{w_\bullet}^+\otimes\varphi^0\otimes\varphi^-:{}_\mathbb{C}\U_P\overset{\sim}{\longrightarrow} \mathcal{O}(P).
\] 
Therefore the algebra $\U_P$ provides a quantization of $\mathcal{O}(P)$.

\subsection{Quantization of the map $\iota^*$}

We next describe a quantization of the map $\iota^*:\mathcal{O}(G^*)\rightarrow\mathcal{O}(P)$.
Note that $\U_{w_\bullet}^+=\U^+\cap\T_{w_\bullet}(\U^-)$. Set $\U^{+\prime}_{w_\bullet}=\U^+\cap \T_{w_\bullet}(\U^+)$ to be another subalgebra of $\U^+$. By the PBW-bases we have the tensor product decomposition $\U^+\cong \U_{w_\bullet}^+\otimes {\U_{w_\bullet} ^{+\prime}}$. Under this isomorphism, set $$\pi_{w_\bullet}^+=id\mid_{\U^+_{w_\bullet}}\otimes\varepsilon:\U^+\longrightarrow\U^+_{w_\bullet}.$$ 

Set $Q'=\mathbb{Z}[\alpha_i\mid i\in \I_\circ']$ to be the sublattice of $Q$. Let $\U^{0}{'}$ be the unital subalgebra of $\U^0$, generated by $K_i^{\pm1}$, for $i\in \I_\circ '$. Then we have the natural isomorphism $\U^0\cong\U^{0\theta}\otimes\U^{0}{}'$. Under this isomorphism, set 
$$\pi^0=id\mid_{\U^{0\theta}}\otimes \varepsilon:\U^0\longrightarrow\U^{0\theta}.$$


Finally, under the isomorphisms \eqref{eq:upd}, set 
\begin{equation}\label{eq:pp}
    \pi=\pi_{w_\bullet}^+\otimes\pi^0\otimes id\mid_{\U^-}:\U\cong\U^+\otimes\U^0\times\U^-\longrightarrow \U_P.
\end{equation}
The map $\pi$ restricts to integral forms $_\A\pi:{}_\A\U\rightarrow {}_\A\U_P$. 

The following proposition asserts that the map $_\A\pi$ quantize the map $\iota^*$.

\begin{prop}\label{prop:cp}
    Let $_\mathbb{C}\pi:{}_\mathbb{C}\U\rightarrow{}_\mathbb{C}\U_P$ be the base change of the map $_\A\pi$. Then the diagram 
    \begin{equation}\label{eq:diag}
        \begin{tikzcd}
          & \mathcal{O}(G^*) \arrow[r,"\iota^*"] & \mathcal{O}(P) \\ & \CU \arrow[u,"\varphi"] \arrow[r,"_\mathbb{C}\pi"'] & \CU_P \arrow[u,"\varphi_P"'] 
        \end{tikzcd}
    \end{equation}
    commutes. 
\end{prop}

\begin{proof}
    The maps $\varphi$, $\iota^*$ and $\varphi_P$ are clearly algebra homomorphisms. Since $_\mathbb{C}\U$ is commutative, the map $_\mathbb{C}\pi$ is also an algebra homomorphism. Therefore it will suffice to check the diagram for generators of $_\mathbb{C}\U$.

    Take a reduced expression $\mathbf{i}=(i_1,i_2,\cdots, i_n)$ of $w_0$, such that $(i_1,i_2,\cdots, i_{n'})$ is a reduced expression of $w_\bullet$. The elements $\overline{E_{\mathbf{i},k}}$, $\overline{F_{\mathbf{i},k}}$, and $\overline{K_i}$, for $1\leq k\leq n$ and $i\in\I$, generate the $\mathbb{C}$-algebra $_\mathbb{C}\U$. It will suffice to check the diagram when acting on these generators, which is direct and will be omitted.
\end{proof}

\subsection{Poisson generators}\label{sec:pg}

The last ingredient that is needed for the proof is a set of Poisson generators of the algebra $\mathcal{O}(K^\perp\backslash G^*)$. 

Let us consider the following elements in $\mathcal{O}(G^*)$:
\begin{itemize}
    \item[(i)] $\chi_i^+$, $\chi_i^-$, for $i\in\I_\bullet$;
    \item[(ii)] $\chi_i^--c_i\T_{w_\bullet}(\chi_{\tau i}^+)\alpha_i^{-1}$, for $i\in\I_\circ$;
    \item[(iii)] $\alpha_i^{\pm1}$, for $i\in\I_\bullet$;
    \item[(iv)] $(\alpha_i\alpha_{\tau i}^{-1})^{\pm1}$, for $i\in\I_\circ-\I_\circ'$.
\end{itemize}

\begin{lemma}\label{le:pg}
    The elements in (i)--(iv) belong to $\mathcal{O}(K^\perp\backslash G^*)$. Moreover they form a set of Poisson generators of $\mathcal{O}(K^\perp\backslash G^*)$. 
\end{lemma}

\begin{proof}
    Let $\pi_2:G^*\rightarrow P$ be the projection map onto $P$ under the isomorphism \eqref{eq:dec}. By Proposition \ref{prop:dec}, the image of the comorphism $\pi_2^*:\mathcal{O}(P)\rightarrow\mathcal{O}(G^*)$ is exactly $\mathcal{O}(K^\perp\backslash G^*)$. 
    
    In order to compute images of various functions under $\pi_2^*$, we give a precise description on the map $\pi_2$. Take any element $(u_+t,t^{-1}u_-)$ in $G^*$. Decompose $u_+=u_+'u_+''$, with $u_+'\in U^+_{w_0w_\bullet}$ and $u_+''\in U^+_{w_\bullet}$. Decompose $t=t't''$, with $t'\in H_0$ and $t''\in H'$. It is direct to verify that 
\begin{equation}\label{eq:p2}
    \pi_2\big((u_+t,t^{-1}u_-)\big)=\big(\text{Ad}_{t'^{-1}}(u_+'')t'',t''^{-1}\text{Ad}_t(\theta(u_+'))^{-1}u_-\big).
\end{equation} 

Let us identify regular functions on $U^+$, $U_{w_\bullet}^-$ and $H'$ as regular functions on $P$, via the isomorphism $P\cong U^+_{w_\bullet}\times H'\times U^-$ \eqref{eq:P}. 

{\it (a) For $i\in \I_\bullet$, we have $\pi_2^*(\chi_{i}^+\mid_{U_{w_\bullet}^+})=\chi_i^{+}$ and $\pi_2^*(\chi_{i}^-)=\chi_i^{-}$.}

Take $i\in\I_\bullet$, and $(u_+t,t^{-1}u_-)\in G^*$. By \eqref{eq:p2}, we have 
    \[
    \pi_2^*(\chi_{i}^+\mid_{U_{w_\bullet}^+})\big((u_+t,t^{-1}u_-)\big)=\chi_{i}^+\big(\text{Ad}_{t'^{-1}}(u_+'')\big)=\alpha_i(t'^{-1})\chi_i^+(u_+).
    \]
    Note that $\alpha_i(t')=1$, because $\theta(\alpha_i)=\alpha_i$ and $\theta(t')=t'^{-1}$. Therefore we have $\pi_2^*(\chi_{P,i}^+\mid_{U_{w_\bullet}^+})=\chi_i^+$. 
    
    Since $\theta(U_{w_0w_\bullet}^+)\cap U^-_{w_\bullet}=\{e\}$, we deduce that
    \[
    \pi_2^*(\chi_{i}^-)\big((u_+t,t^{-1}u_-)\big)=\chi_{i}^-\big(\text{Ad}_t(\theta(u_+'))^{-1}u_-\big)=\chi_i^{-}(u_-).
    \]
    This implies $\pi_2^*(\chi_{i}^-)=\chi_i^-$.

{\it (b) For $i\in \I_\circ$, we have $\pi_2^*(\chi_{P,i}^-)=\chi_i^--c_i\T_{w_\bullet}(\chi_{\tau i}^+)\alpha_i^{-1}$.}

Take $i\in\I_\circ$, and $(u_+t,t^{-1}u_-)\in G^*$. Then by \eqref{eq:p2} one has
    \[
    \pi_2^*(\chi_{i}^-)\big((u_+t,t^{-1}u_-)\big)=\chi_{i}^-\big(\text{Ad}_t(\theta(u_+'))^{-1}u_-\big)=\chi_i^-(u_-)+\chi_i^-(\text{Ad}_t(\theta(u_+'^{-1}))).
    \]
    Take a reduced expression $\mathbf{i}=(i_1,i_2,\cdots, i_n)$ of $w_0$, such that $(i_1,i_2,\cdots, i_{n'})$ is a reduced expression of $w_\bullet$, and $i_{n'+1}=\tau i$. Write $u_+'=x_{\mathbf{i},n}(a_n)x_{\mathbf{i},n-1}(a_{n-1})\cdots x_{\mathbf{i},1}(a_1)$, for $a_k\in\mathbb{C}$. Then $u_+'\in U_{w_0w_\bullet}^+$. Hence we have $a_k=0$ for $1\leq k\leq n'$. Therefore 
    \begin{equation*}
    \begin{split}
    \theta(u_+'^{-1})&=\theta(x_{\mathbf{i},n'+1}(a_{n'+1}))\theta(x_{\mathbf{i},n'+2}(a_{n'+2}))\cdots \theta(x_{\mathbf{i},n}(a_n))\\
    &=y_i(c_ia_{n'+1})\theta(x_{\mathbf{i},n'+2}(a_{n'+2}))\cdots \theta(x_{\mathbf{i},n}(a_n))
    \end{split}
    \end{equation*}
    is a factorisation as in \eqref{eq:fac}. Hence 
    \[
    \chi_i^-(\text{Ad}_t(\theta(u_+'^{-1})))=\chi_i^-(y_i(\alpha_i^{-1}(t)c_ia_{n'+1}))=c_i\alpha_i^{-1}(t)a_{n'+1}=c_i\alpha_i^{-1}(t)\T_{w_\bullet}(\chi_{\tau i}^+)(u_+).
    \]
    This implies that $\pi_2^*(\chi_{P,i}^-)=\chi_i^--c_i\T_{w_\bullet}(\chi^+_{\tau i})\alpha_i^{-1}$.

    {\it (c) For $i\in\I_\bullet$, we have $\pi_2^*(\alpha_{i}\mid_{H'})=\alpha_i$. For $i\in\I_\circ-\I_\circ'$, we have $\pi_2^*(\alpha_{i}\mid_{H'})=\alpha_i\alpha_{\tau i}^{-1}A_i$, where $A_i$ is a monomial of the form $\prod_{i\in\I_\bullet}\alpha_i^{n_i}$ with $n_i\in\mathbb{Z}$.}

    It follows from the similar argument. Statement (c) is actually easier to prove since we only need to consider elements in the torus. We skip the details.

    Thanks to the statements (a)--(c), we conclude that elements in (i)--(iv) belong to the image of $\pi_2^*$ and hence belong to $\mathcal{O}(K^\perp\backslash G^*)$. 
    
    Next we show that these elements form a set of Poisson generators.

    Let $\mathcal{O}'$ be the Poisson subalgebra of $\mathcal{O}(G^*)$ generated by elements in (i)--(iv). By the previous argument we have $\mathcal{O}'\subset\mathcal{O}(K^\perp\backslash G^*)$. Recall that the restriction of the map $\iota^*:\mathcal{O}(G^*)\rightarrow\mathcal{O}(P)$ to $\mathcal{O}(K^\perp\backslash G^*)$ is an isomorphism. In order to show $\mathcal{O}'=\mathcal{O}(K^\perp\backslash G^*)$, it suffices to show that $\iota^*(\mathcal{O}')=\mathcal{O}(P)$. Still let us identify $\mathcal{O}(U_{w_\bullet}^+)$, $\mathcal{O}(H')$ and $\mathcal{O}(U^-)$ as the subalgebras of $\mathcal{O}(P)$. It suffices to show that $\iota^*(\mathcal{O}')$ contains these subalgebras.

Firstly, by the proof of \cite{DCKP}*{Theorem 7.6}, the algebra $\mathcal{O}(U_{w_\bullet}^+)$ is generated by functions $\chi_{i}^+\mid_{U_{w_\bullet}^+}$, for $i\in\I_\bullet$, as a Poisson algebra. Therefore as subalgebras of $\mathcal{O}(G^*)$, we have $\mathcal{O}'\supset \mathcal{O}(U_{w_\bullet}^+)$, thanks to (a). This implies $\iota^*(\mathcal{O}')\supset \mathcal{O}(U_{w_\bullet})$, where $\mathcal{O}(U_{w_\bullet})$ on the right hand side is viewed as a subalgebra of $\mathcal{O}(P)$. 

Next, the algebra $\mathcal{O}(H')$ is generated by $(\alpha_{i}\mid_{H'})^{\pm1}$, for $i\in\I-\I_\circ'$. Thanks to (c) we conclude that $\iota^*(\mathcal{O}')$ contains $\mathcal{O}(H')$.

Lastly, still by the proof of \cite{DCKP}*{Theorem 7.6}, as a Poisson algebra $\mathcal{O}(U^-)$ is generated by $\chi_i^-$, for $i\in\I$. Write $b_i=\chi_i^-$, for $i\in\I_\bullet$, and $b_i=\chi_i^--c_i\T_{w_\bullet}(\chi_{\tau i}^+)\alpha_i^{-1}$, for $i\in\I_\circ$, to be elements in $\mathcal{O}(G^*)$. Let $f(x_1,x_2,\cdots,x_k)$ be a Poisson polynomial, that is, a polynomial possibly with Poisson brackets. Note that there is a $\mathbb{Z}[\I]$-grading on the algebra $\mathcal{O}(G^*)$ respecting the Poisson brackets, where $\text{deg}(\chi_i^-)=i$, $\text{deg}(\chi_i^+)=-i$ and $\text{deg}(\mu)=0$, for $i\in\I$ and $\mu\in Q$. We define a partial order on $\mathbb{Z}[\I]$ by setting $\mu\leq\mu'$ if $\mu'-\mu$ is a non-negative combination of various $i\in \I$. Then for any $i_1,i_2,\cdots,i_k$ in $ \I$, we have $$f(b_{i_1},b_{i_2},\cdots,b_{i_k})=f(\chi_{i_1}^-,\chi_{i_2}^-,\cdots,\chi_{i_k}^-)+\text{lower terms}$$
    in $\mathcal{O}(G^*).$ Therefore we deduce that $\iota^*(\widetilde{\mathcal{O}})$ contains $\mathcal{O}(U^-)\subset\mathcal{O}(P)$ by induction on the degree. 
    
    We complete the proof of the lemma.
\end{proof}

\subsection{The final
proof}

We are now ready to prove the main theorem of the paper.

\begin{proof}[Proof the Theorem \ref{main}]
    The uniqueness of $\varphi^\imath$ is clear. We show the existence. Consider the following diagram
\begin{equation}\label{diag}
    \begin{tikzcd}
            & \mathcal{O}(K^\perp\backslash G^*) \arrow[r,hook] & \mathcal{O}(G^*) \arrow[r,"\iota^*"] & \mathcal{O}(P) \arrow[l,bend left,"\pi_2^*"]\\
            & _\mathbb{C}\U^\imath \arrow[u,"\varphi^\imath"] \arrow["_\mathbb{C}\iota"',r] & {}_\mathbb{C}\U \arrow[u,"\varphi"'] \arrow[r,"_\mathbb{C}\pi"'] & _\mathbb{C}\U_P \arrow[u,"\varphi_P"']
        \end{tikzcd}
\end{equation}

Let $\widetilde{\mathcal{O}}\subset\mathcal{O}(G^*)$ be the image of the map $\varphi\circ{}_\mathbb{C}\iota:{}_\mathbb{C}\U^\imath\rightarrow\mathcal{O}(G^*)$. 

{\it (a) We have $\widetilde{\mathcal{O}}\supset{}\mathcal{O}(K^\perp\backslash G^*)$.}

Recall the generators $B_i$ ($i\in\I$), $E_i$ ($i\in\I_\bullet$), $K_i^{\pm1}$ ($i\in\I_\bullet$) and $k_i$ ($i\in\I_\bullet$) of the subalgebra $\U^\imath$ in \S \ref{sec:SD}. It is clear that these elements belong to the integral form $_\A\U^\imath$. Moreover these elements specialise exactly to the elements (i)--(iv) in \S \ref{sec:pg}. Since the maps $_\mathbb{C}\iota$ and $\varphi$ are Poisson, $\widetilde{\mathcal{O}}$ is a Poisson subalgebra of $\mathcal{O}(G^*)$. Then (a) follows from Lemma \ref{le:pg}.  

{\it (b) The map $_\mathbb{C}\pi\circ{}_\mathbb{C}\iota:{}_\mathbb{C}\U\rightarrow{}_\mathbb{C}\U_P$ is surjective.}

Recall that $\varphi_P$ is an isomorphism. Therefore it suffices to show that $\varphi_P\circ{}_\mathbb{C}\pi\circ{}_\mathbb{C}\iota$ is surjective. By Proposition \ref{prop:cp}, we have $\varphi_P\circ{}_\mathbb{C}\pi\circ{}_\mathbb{C}\iota=\iota^*\circ\varphi\circ{}_\mathbb{C}\iota$. Recall that $\iota^*({}\mathcal{O}(K^\perp\backslash G^*))=\mathcal{O}(P)$. Then (b) follows from (a).

{\it (c) The map $_\mathbb{C}\pi\circ{}_\mathbb{C}\iota:{}_\mathbb{C}\U\rightarrow{}_\mathbb{C}\U_P$ is an isomorphism.}

Since $_\A\U_P$ is a free $\A$-module, we take an $\A$-basis $\mathrm{B_P}$ of $_\A\U_P$. Then $\mathrm{B}_P$ is a $\Qq$-basis of $\U_P$. Thanks to \cite{KY}*{Lemma 2.10}, the Letzter's map $\pi^\imath=\pi\mid_{\U^\imath}:\U^\imath\rightarrow \U_P$ is an isomorphism as vector spaces. Therefore, the set $(\pi^\imath)^{-1}(\mathrm{B}_P)\subset\U^\imath$ gives a $\Qq$-basis of $\U^\imath$. For $b\in \mathrm{B}_P$, let $d_b\in\mathbb{C}[q^{1/2}]$ be the monic polynomial with the lowest degree, such that $d_b(\pi^\imath)^{-1}(b)\in{}_\A\U^\imath$. Then $\{d_b(\pi^\imath)^{-1}(b)\mid b\in\mathrm{B}\}$ is an $\A$-basis of $_\A\U^\imath$. Hence the map $_\mathbb{C}\pi^\imath={}_\mathbb{C}\pi\circ{}_\mathbb{C}\iota:{}_\mathbb{C}\U^\imath\rightarrow{}_\mathbb{C}\U_P$, obtained by the base change of the map $\pi^\imath\mid_{{}_\A\U^\imath}:{}_\A\U^\imath\rightarrow{}_\A\U_P$, is an isomorphism if $d_b(1)\neq 0$, for any $b\in \mathrm{B}$. Suppose otherwise $d_b(1)=0$ for some $b\in \mathrm{B}$. Then $\overline{b}\in{}_\mathbb{C}\U_P$ does not belong to the image of $_\mathbb{C}\pi^\imath$, which contradicts with the surjectivity established in (b). This completes the proof of (c).

{\it (d) We have $\widetilde{\mathcal{O}}={}\mathcal{O}(K^\perp\backslash G^*)$. Hence there exists $\varphi^\imath$ which makes the diagram \eqref{diag} commutes.}

By (c) and the Proposition \ref{prop:cp}, we deduce that $\iota^*\circ\varphi\circ{}_\mathbb{C}\iota=\varphi_P\circ{}_\mathbb{C}\pi\circ{}_\mathbb{C}\iota$ is an isomorphism. In particular $\iota^*\mid_{\widetilde{\mathcal{O}}}$ is injective. Since $\iota^*\mid_{{}^K\mathcal{O}(G^*)}$ is an isomorphism, (d) follows from (a).

{\it (e) The map $\varphi^\imath$ is an isomorphism as Poisson algebras.}

By the commuting diagram \eqref{diag}, we have $\varphi^\imath=\pi_2^*\circ\varphi_P\circ{}_\mathbb{C}\pi\circ{}_\mathbb{C}\iota$. Since the maps $\pi_2^*$, $\varphi_P$ and $_\mathbb{C}\pi\circ{}_\mathbb{C}\iota$ are isomorphisms, we deduce $\varphi^\imath$ is an isomorphism. It is moreover Poisson since $_\mathbb{C}\iota$ and $\varphi$ are Poisson.

We complete the proof of the theorem.
\end{proof}

\section{Example: Dubrovin--Ugaglia Poisson structures}\label{ap}

Let $G=\text{PSL}_n(\mathbb{C})$ and $\theta:G\rightarrow G$ be the map given by $g\mapsto{}^Tg^{-1}$. Then $K=\text{PSO}_n(\mathbb{C})$, and 
\[
K^\perp=\{(ut,{}^Tu^{-1}t^{-1}){}\mid u\in U^+,t\in H\}.
\]
One has an isomorphism as varieties
\begin{equation}\label{du}
K^\perp\backslash G^*\overset{\sim}{\longrightarrow} U^+;\quad K^\perp(tu_+,t^{-1}u_-)\mapsto {}^Tu_-u_+.
\end{equation}

The unipotent subgroup $U^+$ consists of $(n\times n)$-upper triangular matrices with 1 on the diagonal. When viewing $U^+$ as the space of Stokes matrices in the study of Frobenius manifolds, it is naturally equipped with a Poisson structure, called the Dubrovin--Ugaglia Poisson structure (\cites{Du,Ug}). Boalch \cite{Bo} relates such Poisson structure with the stable locus Poisson structure of $G^*$, and then Xu \cite{Xu} interprets the Poisson structure as a Poisson homogeneous space of $G^*$. By \cite{Xu}*{Theorem 5.12}, up to rescaling of Poisson brackets, \eqref{du} is an isomorphism as Poisson manifolds .

Let $\U^\imath$ be the $\imath$quantum group associated to the symmetric pair $(\mathfrak{sl}_n,\mathfrak{so}_n)$ and let $_\mathbb{C}\U^\imath$ be the $\mathbb{C}$-algebra obtained by the base change as before. By Theorem \ref{main} and the isomorphism \eqref{du}, we get the Poisson algebra isomorphism 
\begin{equation}\label{eq:du}
    _\mathbb{C}\U^\imath\overset{\sim}{\longrightarrow} \mathcal{O}(U^+).
\end{equation}

It is also known that the braid group acts naturally on the space $U^+$ via Poisson automorphisms \cites{Du,Ug}. On the other hand, by the work \cite{WZ} the $\imath$quantum groups also admit braid group actions, generalising Lusztig's braid group actions on the quantum groups. Thanks to \cite{So}*{Proposition 3.6}, for the $\imath$quantum group associated to $(\mathfrak{sl}_n,\mathfrak{so}_n)$, the braid group action actually preserves the integral form $_\A\U^\imath$. Therefore it induces the action on $_\mathbb{C}\U$. It is then direct to verify that the isomorphism \eqref{eq:du} is moreover equivariant with respect to the braid group actions. Therefore the braid group symmetries on the Dubrovin--Ugaglia Poisson structure can be understood from the theory of quantum symmetric pairs.


\begin{thebibliography}{}

\bib{Bo}{article}{
   author={Boalch, P. P.},
   title={Stokes matrices, Poisson Lie groups and Frobenius manifolds},
   journal={Invent. Math.},
   volume={146},
   date={2001},
   number={3},
   pages={479--506},
   issn={0020-9910},
   review={\MR{1869848}},
   doi={10.1007/s002220100170},
}

\bib{Bo2}{article}{
   author={Boalch, P. P.},
   title={$G$-bundles, isomonodromy, and quantum Weyl groups},
   journal={Int. Math. Res. Not.},
   date={2002},
   number={22},
   pages={1129--1166},
   issn={1073-7928},
   review={\MR{1904670}},
   doi={10.1155/S1073792802111081},
}

\bib{Bon}{article}{
   author={Bondal, A. I.},
   title={A symplectic groupoid of triangular bilinear forms and the braid
   group},
   language={Russian, with Russian summary},
   journal={Izv. Ross. Akad. Nauk Ser. Mat.},
   volume={68},
   date={2004},
   number={4},
   pages={19--74},
   issn={1607-0046},
   translation={
      journal={Izv. Math.},
      volume={68},
      date={2004},
      number={4},
      pages={659--708},
      issn={1064-5632},
   },
   review={\MR{2084561}},
}

\bib{BG17}{article}{
   author={Berenstein, A.},
   author={Greenstein, J.},
   title={Double canonical bases},
   journal={Adv. Math.},
   volume={316},
   date={2017},
   pages={381--468},
   issn={0001-8708},
   review={\MR{3672910}},
}

\bib{BGM}{article}{
   author={Ballesteros, A.},
   author={Gutierrez-Sagredo, I.},
   author={Mercati, F.},
   title={Coisotropic Lie bialgebras and complementary dual Poisson
   homogeneous spaces},
   journal={J. Phys. A},
   volume={54},
   date={2021},
   number={31},
   pages={Paper No. 315203, 27},
   issn={1751-8113},
   review={\MR{4294868}},
}

\bib{BS22}{article}{
title={Symmetric subgroup schemes, Frobenius splittings, and quantum symmetric pairs}, 
      author={H. Bao and J. Song},
      year={2022},
      eprint={arXiv:2212.13426},
}

\bib{BS24}{article}{
title={Coordinate rings on symmetric spaces}, 
      author={H. Bao and J. Song},
      year={2024},
      eprint={arXiv:2402.08258},
}

\bib{BW18}{article}{
   author={Bao, H.},
   author={Wang, W.},
   title={Canonical bases arising from quantum symmetric pairs},
   journal={Invent. Math.},
   volume={213},
   date={2018},
   number={3},
   pages={1099--1177},
   issn={0020-9910},
   review={\MR{3842062}},
}

\bib{BW}{article}{
AUTHOR = {Bao, H.},
author = {Wang, W.},
     TITLE = {A new approach to {K}azhdan-{L}usztig theory of type {$B$} via
              quantum symmetric pairs},
   JOURNAL = {Ast\'{e}risque},
      YEAR = {2018},
    NUMBER = {402},
     PAGES = {vii+134},
      ISSN = {0303-1179,2492-5926},
      ISBN = {978-2-85629-889-3},
}

\bib{CP}{book}{
   author={Chari, V.},
   author={Pressley, A.},
   title={A guide to quantum groups},
   publisher={Cambridge University Press, Cambridge},
   date={1994},
   pages={xvi+651},
   isbn={0-521-43305-3},
   review={\MR{1300632}},
}

\bib{CG06}{article}{
   author={Ciccoli, N.},
   author={Gavarini, F.},
   title={A quantum duality principle for coisotropic subgroups and Poisson
   quotients},
   journal={Adv. Math.},
   volume={199},
   date={2006},
   number={1},
   pages={104--135},
   issn={0001-8708},
   review={\MR{2187400}},
}

\bib{CG14}{article}{
   author={Ciccoli, N.},
   author={Gavarini, F.},
   title={A global quantum duality principle for subgroups and homogeneous
   spaces},
   journal={Doc. Math.},
   volume={19},
   date={2014},
   pages={333--380},
   issn={1431-0635},
   review={\MR{3178245}},
}

\bib{Du}{article}{
title={Geometry of 2d topological field theories}, 
      author={Dubrovin, B.},
      year={1994},
      eprint={arXiv:hep-th/9407018},
}

\bib{DCKP}{article}{
   author={De Concini, C.},
   author={Kac, V. G.},
   author={Procesi, C.},
   title={Quantum coadjoint action},
   journal={J. Amer. Math. Soc.},
   volume={5},
   date={1992},
   number={1},
   pages={151--189},
   issn={0894-0347},
   review={\MR{1124981}},
}

\bib{Dr86}{article}{
   author={Drinfel\cprime d, V. G.},
   title={Quantum groups},
   conference={
      title={Proceedings of the International Congress of Mathematicians,
      Vol. 1, 2},
      address={Berkeley, Calif.},
      date={1986},
   },
   book={
      publisher={Amer. Math. Soc., Providence, RI},
   },
   isbn={0-8218-0110-4},
   date={1987},
   pages={798--820},
   review={\MR{0934283}},
}

\bib{Dr93}{article}{
   author={Drinfel\cprime d, V. G.},
   title={On Poisson homogeneous spaces of Poisson-Lie groups},
   language={English, with English and Russian summaries},
   journal={Teoret. Mat. Fiz.},
   volume={95},
   date={1993},
   number={2},
   pages={226--227},
   issn={0564-6162},
   translation={
      journal={Theoret. and Math. Phys.},
      volume={95},
      date={1993},
      number={2},
      pages={524--525},
      issn={0040-5779},
   },
   review={\MR{1243249}},
}



\bib{Jan96}{book}{
   author={Jantzen, J. C.},
   title={Lectures on quantum groups},
   series={Graduate Studies in Mathematics},
   volume={6},
   publisher={American Mathematical Society, Providence, RI},
   date={1996},
   pages={viii+266},
   isbn={0-8218-0478-2},
   review={\MR{1359532}},
}

\bib{KY}{article}{
   author={Kolb, S.},
   author={Yakimov, M.},
   title={Defining relations of quantum symmetric pair coideal subalgebras},
   journal={Forum Math. Sigma},
   volume={9},
   date={2021},
   pages={Paper No. e67, 38},
   review={\MR{4321012}},
}


\bib{Let02}{article}{
   author={Letzter, G.},
   title={Coideal subalgebras and quantum symmetric pairs},
   conference={
      title={New directions in Hopf algebras},
   },
   book={
      series={Math. Sci. Res. Inst. Publ.},
      volume={43},
      publisher={Cambridge Univ. Press, Cambridge},
   },
   isbn={0-521-81512-6},
   date={2002},
   pages={117--165},
   review={\MR{1913438}},
}

\bib{Lu93}{book}{
   author={Lusztig, G.},
   title={Introduction to quantum groups},
   series={Progress in Mathematics},
   volume={110},
   publisher={Birkh\"{a}user Boston, Inc., Boston, MA},
   date={1993},
   pages={xii+341},
   isbn={0-8176-3712-5},
   review={\MR{1227098}},
}

\bib{Sh22}{article}{
title={Cluster Nature of Quantum Groups}, 
      author={Shen, L.},
      year={2022},
      eprint={arXiv:2209.06258},
}

\bib{So}{article}{
title={Cluster realisations of $\imath$quantum groups of type AI}, 
      author={J. Song},
      year={2023},
      eprint={arXiv:2311.03786},
}


\bib{Spr}{article}{
   author={Springer, T. A.},
   title={The classification of involutions of simple algebraic groups},
   journal={J. Fac. Sci. Univ. Tokyo Sect. IA Math.},
   volume={34},
   date={1987},
   number={3},
   pages={655--670},
   issn={0040-8980},
   review={\MR{0927604}},
}

\bib{Spr09}{book}{
   author={Springer, T. A.},
   title={Linear algebraic groups},
   series={Modern Birkh\"{a}user Classics},
   edition={2},
   publisher={Birkh\"{a}user Boston, Inc., Boston, MA},
   date={2009},
   pages={xvi+334},
   isbn={978-0-8176-4839-8},
   review={\MR{2458469}},
}

\bib{SS}{article}{
      AUTHOR = {Schrader, G.},
      author = {Shapiro, A.}
     TITLE = {A cluster realization of {$U_q(\germ {sl}_{\germ n})$} from
              quantum character varieties},
   JOURNAL = {Invent. Math.},
    VOLUME = {216},
      YEAR = {2019},
    NUMBER = {3},
     PAGES = {799--846},
      ISSN = {0020-9910,1432-1297},
}

\bib{Ug}{article}{
   author={Ugaglia, M.},
   title={On a Poisson structure on the space of Stokes matrices},
   journal={Internat. Math. Res. Notices},
   date={1999},
   number={9},
   pages={473--493},
   issn={1073-7928},
   review={\MR{1692595}},
}

\bib{WZ}{article}{
   author={Wang, W.},
   author={Zhang, W.},
   title={An intrinsic approach to relative braid group symmetries on
   $\iota$quantum groups},
   journal={Proc. Lond. Math. Soc. (3)},
   volume={127},
   date={2023},
   number={5},
   pages={1338--1423},
   issn={0024-6115},
   review={\MR{4668526}},
}

\bib{Xu}{article}{
   author={Xu, P.},
   title={Dirac submanifolds and Poisson involutions},
   language={English, with English and French summaries},
   journal={Ann. Sci. \'{E}cole Norm. Sup. (4)},
   volume={36},
   date={2003},
   number={3},
   pages={403--430},
   issn={0012-9593},
   review={\MR{1977824}},
}

\end{thebibliography}
\end{document}